\documentclass[hidelinks,onefignum,onetabnum]{siamart220329}

\usepackage{amsfonts}
\usepackage{graphicx}
\usepackage{float}
\usepackage{subfig}
\usepackage{xcolor}
\usepackage{amsmath,amsbsy}
\usepackage{mathtools}
\usepackage{tikz}
\usepackage{adjustbox}
\usepackage{amssymb}
\usepackage{tabularx}
\usepackage{hyperref}
\usetikzlibrary{decorations}
\usetikzlibrary{arrows.meta}
\usetikzlibrary{shapes,arrows,chains}
\usetikzlibrary{spy}
\usetikzlibrary{trees}
\usepackage{pgfplots}
\pgfplotsset{compat = 1.18}

\usepackage{algorithmic}
\ifpdf
  \DeclareGraphicsExtensions{.eps,.pdf,.png,.jpg}
\else
  \DeclareGraphicsExtensions{.eps}
\fi

\newsiamremark{remark}{Remark}
\newsiamremark{hypothesis}{Hypothesis}

\newcommand\notsotiny{\@setfontsize\notsotiny\@vipt\@viipt}

\newcommand{\lra}{\longrightarrow}
\newcommand{\Tr}{\operatorname{Tr}}

\newcommand{\Om}{\Omega}
\newcommand{\R}{\mathbb{R}}

\newcommand{\eps}{\varepsilon}

\renewcommand{\div}{\operatorname{div}}
\newcommand{\Rey}{\text{Re}}
\renewcommand{\Cap}{\text{Ca}}
\newcommand{\I}{\mathbb{I}}
\newcommand{\pa}{\partial}
\newcommand\Tstrut{\rule{0pt}{2.6ex}}         % = `top' strut
\newcommand\Bstrut{\rule[-0.9ex]{0pt}{0pt}}   % = `bottom' strut

\headers{Semi-implicit Eulerian method for fluid-structure interaction}{M. Ciallella, and T. Milcent}

\title{Semi-implicit Eulerian method for the fluid structure interaction of elastic membranes
\thanks{Submitted to the editors 27/10/2023.\funding{This work was funded by the ANR CAPSEULERIAN-FSI.}}}

\author{Mirco Ciallella\thanks{École Nationale Supérieure d’Arts et Métiers, Institut de Mécanique et d’Ingénierie, 33400 Talence, France
  (\email{mirco.ciallella@ensam.eu}).}
\and Thomas Milcent\thanks{École Nationale Supérieure d’Arts et Métiers, Institut de Mécanique et d’Ingénierie, 33400 Talence, France
  (\email{thomas.milcent@ensam.eu}).}}

\usepackage{amsopn}

\begin{document}

\maketitle

% REQUIRED
\begin{abstract}
In this paper we propose a novel and general approach to design semi-implicit methods for the simulation of fluid-structure interaction problems in a fully Eulerian framework. 
In order to properly present the new method, we focus on the two-dimensional version of the general model developed to describe full membrane elasticity. 
The approach consists in treating the elastic source term by writing an evolution equation on the structure stress tensor, even if it is nonlinear.
Then, it is possible to show that its semi-implicit discretization allows us to add to the linear system of the Navier-Stokes equations some consistent dissipation terms that depend on the local deformation and stiffness of the membrane.
Due to the linearly implicit discretization, the approach does not need iterative solvers and can be easily applied to any Eulerian framework for fluid-structure interaction.
Its stability properties are studied by performing a Von Neumann analysis on a simplified one-dimensional model and proving that, thanks to the additional dissipation, the discretized coupled system is unconditionally stable.
Several numerical experiments are shown for two-dimensional problems by comparing the new method to the original explicit scheme and studying the effect of structure stiffness and mesh refinement on the membrane dynamics.
The newly designed scheme is able to relax the time step restrictions that affect the explicit method and reduce crucially the computational costs, especially when very stiff membranes are under consideration.
\end{abstract}

% REQUIRED

% REQUIRED

\section{Introduction}

Fluid-structure interaction problems arise in a wide range of applications and their numerical simulation is extremely challenging.
In biomechanics, the correct modeling and computation of membranes immersed in an incompressible fluid is a key point to study biological
capsules and cells, like red blood cells (RBCs). Capsules are in general modeled as a liquid drop protected by a thin elastic membrane
and their use is widely spread in food, cosmetics and pharmaceutical industries.
Simulations of such fluid-structure problems are expensive and require the correct coupling between the flow behavior and the structure behavior,
which often comes with the imposition of some kind of interface conditions. One of the most classical approaches is the so-called 
Arbitrary Lagrangian-Eulerian (ALE) method~\cite{liu1998numerical,legay2006eulerian,sahin2009arbitrary,donea1982arbitrary,fanion2000deriving}. 
This approach consists in solving the fluid-structure problem on a moving mesh that follows the displacement
of the interface, and enforce the coupling conditions right on the grid points that describe the interface.
Although this philosophy may bring accurate results for some problems, its limitations are clear when dealing with large deformations or topology changes.

In~\cite{peskin1972flow}, Peskin introduced a new approach that simplifies the fluid-structure coupling for applications with structures modeled as lower dimensional manifolds of the problem.
For immersed boundaries~\cite{peskin2002immersed,lee2003immersed,mittal2005immersed}, the coupling interface conditions reduce to a forcing term in the Navier-Stokes equations that modifies the flow behavior by taking 
into account the deformation of the immersed membrane. In this case, the structure is tracked using Lagrangian markers and the source term is spread on the fluid with a discretized Dirac delta.  
When large deformations are under consideration the need to insert or delete markers for stretching and shrinking of the membrane remains a delicate aspect, which requires the introduction of additional parameters to tune.

Both ALE and immersed boundary (IB) methods belong to the class of front-tracking methods that need an explicit definition and displacement of the interface.
Fully Eulerian models for fluid-structure interaction have been recently introduced~\cite{cottet2006level} for computational purposes as they allow to simulate quite easily multi-dimensional problems
with large deformations thanks to an implicit treatment of interfaces and deformations through level-set fields~\cite{Osher1988,Sussman1994}. 
In the last decade, Eulerian elasticity has drawn a lot of attention and its implementation turned out to be very successful for both 
incompressible~\cite{cottet2008eulerian,maitre2009applications,richter2013fully,Milcent2016,Deborde2020,bergmann2022eulerian} and compressible~\cite{gorsse2014simple,de2016cartesian,de2017cartesian} applications. 
These models rely on the so-called backward characteristics which are defined as the inverse map of the deformation, and correlates the location of a material point to its initial position.
As mentioned above, this vector is also defined with Eulerian variables that are advected by the Eulerian fluid velocity, and its spatial derivatives are used to compute the elastic stress tensor.
Rather than following the backward characteristics, it was proposed in~\cite{sugiyama2011full} an Eulerian approach based on following the 
evolution of the symmetrical left Cauchy-Green tensor through an inhomogeneous advection equation with extra terms depending 
on the velocity gradient.  In this case, six components in three dimensions (three in two dimensions) must be advected, 
rather than three in three dimensions (two in two dimensions) for the backward characteristics. 

Whether one deals with ALE, IB or fully Eulerian models for incompressible fluid-structure interaction problems, stability is a well-known issue of monolithic frameworks in which both the fluid and the solid are solved at the same time. 
For Eulerian models this was proven to be related to the velocity of elastic waves that (for stiff materials) are order of magnitude higher than that of the fluid, which leads to dramatic time step restrictions~\cite{de2016cartesian}. 
In order to overcome such limitations related to monolithic frameworks, several works have been dedicated to the derivation of implicit schemes~\cite{tu1992stability,mayo1992implicit,stockie1999analysis,newren2007enhancing}.
However, due to the strong nonlinearity of the problem, these schemes require computationally expensive and cumbersome iterative approaches.
With the goal of moving towards methods that are easier to introduce in different computational frameworks, semi-implicit or approximate implicit schemes are a very interesting and 
widespread research topic since their goal is to provide a less expensive and efficient alternative 
to standard implicit methods~\cite{tu1992stability,stockie1999analysis,fernandez2007projection,fernandez2013explicit}.
Similar interest in such schemes was also shown in the context of multi-phase flows for the computation of surface tension~\cite{hysing2006new,sussman2009stable} to avoid the time step restriction proportional to $\Delta x^{3/2}$.

In the context of fully Eulerian models for fluid-structure interaction, the attempts to relax the time step restrictions have been proposed in~\cite{ii2011implicit,cottet2016semi} with methods that seem to be tailored to the application under consideration.
In particular, the first reference~\cite{ii2011implicit} addresses the simulation of solid bulks immersed in an incompressible fluid by introducing a linear expansion for the nonlinear elastic stress with a Jacobian tensor, 
to improve the performances of the scheme initially proposed in~\cite{sugiyama2011full}.
Whereas in~\cite{cottet2016semi} a semi-implicit approach for thin membranes subjected to area variation is developed by 
solving a diffusion equation to predict the interface position, in the context of the model proposed in~\cite{cottet2006level}.
In this work, we aim at proposing a general way of deriving semi-implicit schemes for Eulerian models for fluid-structure interaction 
that can be directly applicable to general (nonlinear) elastic models without any need of linearization.
This approach consists in discretizing the evolution equation of the stress tensor in a semi-implicit way, and writing the value of the stress at the new time level as a function of that at the previous time plus some additional terms. 
In order to show the generality and potential of this approach, we applied it to the two-dimensional version of the full membrane model introduced in~\cite{Milcent2016}, which is capable of dealing with both area and shear variations.
While the level-set approach proposed in~\cite{cottet2016semi} only works for the elastic force that models the area variation 
introduced in~\cite{cottet2006level}, our idea can be ideally applied to more general and nonlinear elastic models, because it depends on the deformation vector. 
Indeed, the model treated in this paper is the starting point to develop semi-implicit schemes for full membrane elasticity that 
includes both area and shear variations~\cite{Milcent2016}. The complete three-dimensional analysis will be tackled in a future work.

The paper is organized as follows. In Section~\ref{sec:EulerianModel}, we recall the basic definitions to derive the fully Eulerian model for fluid-structure interaction focusing on backward characteristics, level-set
and the membrane elastic model treated herein. 
In Section~\ref{sec:stability} we recall and summarize the main stability results obtained for a simplified one-dimensional model, and derive our semi-implicit scheme in this framework. 
Here it is shown that the new scheme consists in adding a consistent dissipation to the system which allows us to prove its unconditional stability. 
In Section~\ref{sec:schemes}, we begin by describing the classical explicit coupling of the elastic source term and the general algorithm to solve the full model. 
Here, we also derive an evolution equation on the stress tensor that will be then used to design the semi-implicit approach.
After that, we focus on the derivation of the semi-implicit coupling which is performed by introducing the discretized evolution equation of the stress in the momentum equation. 
Here we show that also the approach for the full model consists in adding a diffusion that depends on stiffness and area variation.  
In Section~\ref{sec:manufsol}, we test the convergence properties of the discretized semi-implicit operators to show that they are second-order accurate in space.
In Section~\ref{sec:simulations}, we provide numerical illustrations on the ability of the newly developed semi-implicit scheme to introduce an important relaxation on the time step restriction to simulate a membrane immersed in a simple shear flow. 
In particular, we focus on the influence of membrane stiffness and mesh refinement on the computational gain and quality of the solution.
In Section~\ref{sec:conclusions}, we conclude with a summary of the work and give some perspectives on aspects that need further investigation.
Finally, in Appendix~\ref{app:A} we give a proof on the parameter chosen to measure the area variation, 
in Appendix~\ref{app:B} we give the proof of the evolution equation of the stress tensor, and 
in Appendix~\ref{app:C} we provide examples on how to discretize the semi-implicit operators on Cartesian staggered meshes.

\section{Fully Eulerian Model}
\label{sec:EulerianModel}

\subsection{Forward and backward characteristics}

Let $\Om_0\subset \R^2$ be the reference configuration of a continuous medium and assume that this medium is deformed by a smooth map $X:\R^{+}\times \R^{2}\lra \R^2$ (the forward characteristics) to $\Om_t = X(t,\Om_0)$. A velocity field $u:\R^2\times\R^{+}\lra \R^2$ is naturally associated with $X$:
\begin{equation}
\partial_tX(t,\xi) = u (X(t,\xi),t),  \qquad   X(0,\xi)=\xi, \qquad \xi\in\Om_0.
\label{carac_lag}
\end{equation}
We introduce the backward characteristics $Y:\R^2\times \R^{+}\lra \R^2$ by the formula $Y(X(t,\xi),t) = \xi$. The physical interpretation of $Y(x,t)$ is the position at time $0$ of a material particle lying in $x$ at time $t$ and moving at speed $u$. The derivative of this relation with respect to $t$ and $\xi$ in turn gives with (\ref{carac_lag})
\begin{equation}
\partial_tY + (u\cdot \nabla_x) Y = 0,  \qquad  Y(x,0)=x, \qquad x\in\Om_t,
\label{transportY}
\end{equation}
and
\begin{equation}
[\nabla_{\xi}X(t,\xi)] = [\nabla_x Y(x,t)]^{-1}, \quad\text{for }x=X(t,\xi).
\label{grad_XY}
\end{equation}
The relation (\ref{transportY}) is the Eulerian equivalent of the characteristic equation (\ref{carac_lag}). In addition, equation (\ref{grad_XY}) allows to compute the gradient of the deformation in the Eulerian frame via $Y$. 
The next sections are devoted to the description of the mathematical model used to treat the membrane elastic deformations with the backward characteristics $Y$.
To simplify the notation, we drop the subscript of the operator $\nabla_x$.
\begin{figure}
\centering
\tikzset{>=latex}
\begin{tikzpicture}[scale=1.0]
\pgfmathsetseed{1}
\draw (-0.1,2) node[left] {$Y(x,t)=\xi$};
\fill [black] (0,2) circle (2pt);
\draw (5.7,2.5) node[right] {$x=X(t,\xi)$};
\fill [black] (5.5,2.5) circle (2pt);
\draw[->,thick] (5.5,2.5) -- (4.8,3.5);
\draw (4.8,3.7) node {$n(x,t)$};
\draw[->,thick,dashed] (0,2) to[bend left] (5.5,2.5);
\draw[->,thick,dashed] (5.5,2.5) to[bend left] (0,2);
\draw (-2,-1.5) node[below right] {Initial configuration $\Omega_0$};
\draw (5,-1.5) node[below right] {Deformed configuration $\Omega_t$};
\draw[very thick] (0,1.4) circle (60pt);
\draw[very thick] (8,0) .. controls +(0.8,1.2)   and +(0.7,-0.7) ..
(7,3.5)
                         .. controls +(-0.9,0.6) and +(0.5,0.6) .. (5,3)
                         .. controls +(-0.6,-0.6)  and +(-0.4,0.6) ..
(5.5,0)
                         .. controls +(0.4,-0.5)  and +(-0.7,-1.3) ..
(8,0) -- cycle;
\end{tikzpicture}
\caption{Forward and backward characteristics.}\label{fig:characteristics}
\end{figure}
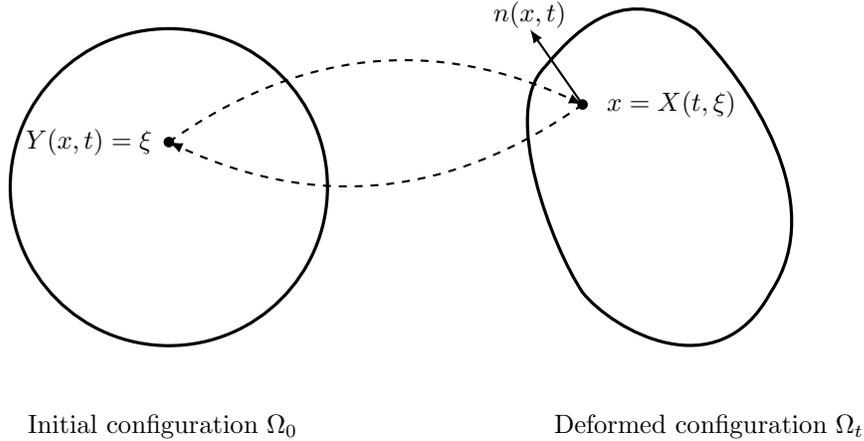
\subsection{Hyperelastic membrane models}

The notations and results summarized in this section are detailed in~\cite{Milcent2016} for the general three-dimensional model. 
We consider a surface $\Gamma_t =\{ x \in \R^2 \; / \; \phi(x,t)=0 \}$ captured by a level set function $\phi:\R^2\times \R^{+}\lra \R$ and advected by the Eulerian velocity field $u$:
\begin{equation}
\partial_t \phi + u\cdot \nabla \phi = 0.
\label{advphi}
\end{equation}
The normal $n(x,t)$ for  $x\in\Gamma_t$ is then expressed in terms of the normalized gradient of the level set:
\begin{equation}
n(x,t) = \frac{\nabla \phi(x,t)}{|\nabla \phi(x,t)|}.
\label{normales}
\end{equation}
To measure the deformations on the surface $\Gamma_t$ we introduce the tensor
\begin{equation}
 \mathcal{A} = B - \frac{(Bn)\otimes(Bn)}{(Bn)\cdot n}.
\label{AB}
\end{equation}
where $B=[\nabla Y]^{-1}[\nabla Y]^{-T}$ is the left Cauchy-Green tensor. Tensor $\mathcal{A}$ measures the surface deformations by projecting the deformations (measured by $B$) on the surface $\Gamma_t$ (represented locally by $n$). The vector $n$ is an eigenvector of $\mathcal{A}$ associated to the eigenvalue $0$ so $\det(\mathcal{A})=0$. In two dimensions, the other invariant (the trace) is used to define 
\begin{equation}
Z = \sqrt{\Tr(\mathcal{A}) }.
\label{Z1_Z2b}
\end{equation}
Following the same approach in~\cite{Milcent2016}, it is possible to prove that $Z$ measures the local area variation for two-dimensional problems (see Appendix~\ref{app:A} for the proof). 
\begin{remark}
Using the backward characteristics $Y$, rather than only the level-set $\phi$ as in~\cite{cottet2006level}, allows us to have the 
deformations that can also be used to measure the shear variation for tree-dimensional membranes. 
Indeed, the transport equation for the level-set function only records information on the surface area variations, and
ignores any tangential component of $u$. This is due to the fact that $\nabla\phi$ is normal to the interface.
\end{remark}
In order to compute the elastic force, we then  introduce the regularized membrane energy,
\begin{equation}
\mathcal{E} = \int_{Q} E(Z)\frac{1}{\eps}\zeta\left(\frac{\phi}{\eps}\right)\;{\rm d}x.
\end{equation}
Here $Q$ is a box containing the membrane, $E(Z)$ is the constitutive law, $\varepsilon$ the width of the interface and $\zeta$ is a cut-off function used to spread the interface near $\{ \phi=0\}$. 
The associated force, computed with the principle of virtual power, is given by
\begin{equation}
F = \div\left(\sigma \right) \delta_\eps(\phi),\quad \text{where} \quad \sigma = E'(Z) Z (\I - n\otimes n) 
\label{eqFm}
\end{equation}
where 
$$\delta_\eps(\phi) = \frac{1}{\eps}\zeta\left(\frac{\phi}{\eps}\right)$$ 
is introduced to simplify the notation.\\
The Evan-Skalak constitutive law is used in this article:
\begin{equation}
	E'(Z) = K (Z-1),
    \label{eq:law_evan_skalak}
\end{equation}
where $K$ is the elastic modulus of the membrane.\\

\subsection{Fluid-structure interaction model}

The elastic membrane is immersed in a incompressible fluid modeled by the Navier-Stokes equations. The overall fully Eulerian model is given by
\begin{equation}
\left\{
\begin{aligned}
\rho ( \pa_t u + (u\cdot\nabla) u ) + \nabla p -\div(2\mu D(u)) &= F(\phi,Y), \\
							\div(u) &= 0,         \\
				\pa_t \phi + u\cdot \nabla \phi &= 0,         \\
				    \pa_t Y + (u\cdot \nabla) Y &= 0, 
\label{eq_tot}
\end{aligned}
\right.
\end{equation}
where $u$ is the fluid velocity, $D(u) = ([\nabla u] +[\nabla u]^T)/2$ the strain rate tensor, $\rho$ the density, $p$ the pressure and $\mu$ the viscosity. 
Different values of density and viscosity, outside (1) and inside (2) the membrane, are taken into account by defining 
$$ \rho = \mathcal H\left(\frac\phi\eps\right)\rho_1 + \left(1- \mathcal H\left(\frac\phi\eps\right) \right) \rho_2, \quad \mu = \mathcal H\left(\frac\phi\eps\right)\mu_1 + \left(1- \mathcal H\left(\frac\phi\eps\right) \right) \mu_2,$$
where $\mathcal H$ represents a smooth Heaviside function.
The membrane force $F$ is given by \eqref{eqFm}. 
%Note that a level set function $\phi$ and two backward characteristics $Y=[Y_1\;\;Y_2]^T$ are needed for the description of a two-dimensional membrane. 
These equations are completed with appropriate initial and boundary conditions that will be detailed in the numerical validation section.\\

The fluid-structure problem is governed by two parameters, the Reynolds number $\Rey$ and the capillary number $\Cap$, 
that can be recovered by rewriting the momentum equation in a non-dimensional form. 
Therefore, we introduce the reference variables $\varphi^\ast$ and the non-dimensional variables $\bar\varphi$, which give the following operators and physical variables:
$$ \rho = \rho^\ast \,\bar\rho, \quad u = u^\ast \,\bar u,\quad p = p^\ast \,\bar p,\quad \mu = \mu^\ast \,\bar \mu, \quad F = \frac{K^\ast}{(L^\ast)^2}\,\bar F, \quad \nabla = \frac{\bar\nabla}{L^\ast}. $$
By replacing all variables with their counterpart, and taking $p^\ast=\mu^\ast u^\ast/L^\ast$, the non-dimensional momentum equation reads,
\begin{equation}\label{eq:momnondim}
\Rey \left[ \bar\rho \left(  \pa_t \bar u + (\bar u \cdot\bar\nabla) \bar u \right) \right]  + \bar\nabla\bar p - \div(\bar\mu ([\bar\nabla \bar u] +[\bar\nabla \bar u]^T)) = \frac{1}{\Cap}\bar F,
\end{equation}
where $\Rey=\rho^\ast u^\ast L^\ast/\mu^\ast$, and $\Cap = \mu^\ast u^\ast / K^\ast$.

\subsection{Reinitialization and extrapolation}
It is a well-known problem that the level-set field $\phi$ stops being a signed distance function after its time evolution with the associated advection equation.
In this work, for its reinitialization, we follow the idea proposed in~\cite{Sussman1994}  and solve for a fictitious time $\tau$
\begin{equation}\label{eq:levelsethamilton}
\partial_\tau \phi + \text{sgn}(\phi_{0})(|\nabla\phi|-1)=0.
\end{equation}
Moreover, as discussed in~\cite{Deborde2020,bergmann2022eulerian} also the backward characteristics $Y$ can get extremely distorted
with time. Being the elastic force computed by means of $Y$ and $\phi$, if not treated, these distortions are likely to cause spurious disturbances 
close to the fluid-structure interface.\\
To summarize previous works on this topic, we perform 
%a diffusion algorithm for the inner characteristics, i.e.\ $\phi<0$,
%
%\begin{equation}\label{eq:inndiffusion}
%\partial_t Y - \left(1- H(\phi/\eps)\right)\div( \nabla Y ) =0,
%\end{equation}
% 
the linear extrapolation proposed in~\cite{Aslam2003} for the outer characteristics, i.e.\ $\phi>0$,
\begin{align}\label{eq:aslam}
&\partial_\tau Y_n + \mathcal H\left(\frac\phi\eps\right) (n\cdot Y_n) = 0, \quad\text{with}\quad Y_n = n\cdot \nabla Y, \\
&\partial_\tau Y   + \mathcal H\left(\frac\phi\eps\right) (n\cdot Y - Y_n) = 0.
\end{align}

\section{Stability analysis for a simplified one-dimensional model}\label{sec:stability}

In this section, we recall recent stability results in the context of explicit and implicit coupling for a simplified model,
and we introduce the new semi-implicit approach in this framework. 
%For a simplified model, it is also possible to prove some theoretical results about the stability of the new method as shown below.
We consider here the same linearization and assumptions taken for similar models~\cite{bost2009linear,fondaneche2021interaction}. 
To summarize we set $\rho=1$ and we decompose the solution as the small perturbation of a stationary solution $(u_0,Y_0)=(0,x)$.\\
By avoiding the treatment of nonlinear terms we end up with the following linearized one-dimensional model:
\begin{equation}
\left\{
\begin{aligned}
&\frac{\partial u}{\partial t} - \mu \frac{\partial^2 u}{\partial x^2} = -\frac{K}{\eps}\frac{\partial^2 Y}{\partial x^2},\\[5pt]
&\frac{\partial Y}{\partial t} + u = 0.
\end{aligned}
\right.
\end{equation}
Here we recall the stability conditions for such models when using a classical explicit coupling similarly to what was introduced in~\cite{bost2009linear}. 
For simplicity we consider a classical centered discretization of second order derivatives, with grid size $\Delta x$, with an implicit treatment for the viscous term:
\begin{equation}
\left\{
\begin{aligned}
&\frac{u_j^{n+1}-u_j^n}{\Delta t} - \mu \frac{u_{j+1}^{n+1}-2u_j^{n+1}+u_{j-1}^{n+1}}{(\Delta x)^2} = -\frac{K}{\eps}\frac{Y_{j+1}^{n}-2 Y_j^{n}+Y_{j-1}^{n}}{(\Delta x)^2}, \\[5pt]
&\frac{Y_j^{n+1}-Y_j^n}{\Delta t} + u_j^{n+1} = 0.
\end{aligned}
\right.
\end{equation}
For this system, the stability condition was proven in~\cite{bost2009linear} to be 
\begin{equation}
\Delta t < \frac{\mu \eps + \max(\mu\eps,\sqrt{K \eps}\Delta x)}{K}.
\end{equation}
Instead, when treating implicitly the elastic term, the discretized system reads:
\begin{equation}
\left\{
\begin{aligned}
&\frac{u_j^{n+1}-u_j^n}{\Delta t} - \mu \frac{u_{j+1}^{n+1}-2u_j^{n+1}+u_{j-1}^{n+1}}{(\Delta x)^2} = -\frac{K}{\eps}\frac{Y_{j+1}^{n+1}-2 Y_j^{n+1}+Y_{j-1}^{n+1}}{(\Delta x)^2}, \\[5pt]
&\frac{Y_j^{n+1}-Y_j^n}{\Delta t} + u_j^{n+1} = 0.
\end{aligned}
\right.
\end{equation}
Once again in~\cite{bost2009linear}, this system was proven to be unconditionally stable. 
However, for real applications, this scheme is very cumbersome to design and implement 
due to the full fluid-structure Navier-Stokes system with nonlinear source term and advection equations.\\
For this reason, following the same motivation of~\cite{cottet2016semi}, we introduce and study the stability analysis of the new semi-implicit scheme in this simplified framework. 
The new approach developed for the full nonlinear model is described in Section~\ref{sec:semiimplicit}. 

We start by introducing an auxiliary variable $\tilde Y$ for the implicit elastic term,
\begin{equation}\label{eq:SIlinearmodel0}
\left\{
\begin{aligned}
&\frac{u_j^{n+1}-u_j^n}{\Delta t} - \mu \frac{u_{j+1}^{n+1}-2u_j^{n+1}+u_{j-1}^{n+1}}{(\Delta x)^2} = -\frac{K}{\eps}\frac{\tilde Y_{j+1}^{n+1}-2 \tilde Y_j^{n+1}+ \tilde Y_{j-1}^{n+1}}{(\Delta x)^2},\\[5pt]
&\frac{Y_j^{n+1}-Y_j^n}{\Delta t} + u_j^{n+1} = 0.
\end{aligned}
\right.
\end{equation}

For the simplified model, writing an evolution equation on the stress tensor, discretized in a semi-implicit manner, entails that~$\tilde Y^{n+1}_j=Y_j^n-\Delta t u_j^{n+1}$. 
This gives us an explicit part and an implicit one that depends on the second order derivatives of $u$, 
$$ \frac{\tilde Y_{j+1}^{n+1}-2 \tilde Y_j^{n+1}+ \tilde Y_{j-1}^{n+1}}{(\Delta x)^2} = \frac{Y_{j+1}^{n}-2 Y_j^{n}+Y_{j-1}^{n}}{(\Delta x)^2} - \Delta t \frac{u_{j+1}^{n+1}-2u_j^{n+1}+u_{j-1}^{n+1}}{(\Delta x)^2}.$$

When introducing the previous relation in System~\eqref{eq:SIlinearmodel0}, the overall system reads

\begin{equation}\label{eq:SIlinearmodel}
\left\{
\begin{aligned}
&\frac{u_j^{n+1}-u_j^n}{\Delta t} - \left(\mu + \Delta t \frac{K}{\eps}\right) \frac{u_{j+1}^{n+1}-2u_j^{n+1}+u_{j-1}^{n+1}}{(\Delta x)^2} = -\frac{K}{\eps}\frac{Y_{j+1}^{n}-2 Y_j^{n}+Y_{j-1}^{n}}{(\Delta x)^2}, \\[5pt]
&\frac{Y_j^{n+1}-Y_j^n}{\Delta t} + u_j^{n+1} = 0.
\end{aligned}
\right.
\end{equation}

It should be noticed that this approach reduces to solving a PDE similar to the one arising from the explicit discretization, 
but with an additional consistent viscosity term $\Delta t\frac{K}{\eps}$ that depends on the time step and the elastic parameter of the structure.

\begin{proposition}\label{prop:stabilitySI}
The semi-implicit scheme~\eqref{eq:SIlinearmodel} is unconditionally stable.
\end{proposition}

\begin{proof}
With the assumption of periodic boundary conditions we perform a Von-Neumann stability analysis.
Therefore, we decompose the solution $(u,Y)$ as Fourier expansions of the grid values,
{\small
\begin{equation*}
u_j^n = \hat u^n(k)\text{e}^{2i\pi kj\Delta x}, Y_j^n = \hat Y^n(k)\text{e}^{2i\pi kj\Delta x}\quad \xrightarrow{\theta:=2\pi k \Delta x}\quad  u_j^n = \hat u^n(\theta)\text{e}^{ij\theta}, Y_j^n = \hat Y^n(\theta)\text{e}^{ij\theta}.
\end{equation*}
}
We then obtain from~\eqref{eq:SIlinearmodel} 
{\small
\begin{equation*}
\left\{
\begin{aligned}
&\hat u^{n+1}(\theta)\left[1 + 4\frac{\Delta t}{(\Delta x)^2}\left(\mu + \Delta t \frac{K}{\eps}\right)\sin^2\left(\frac\theta2\right)\right] = \hat u^n(\theta) + \hat Y^n(\theta) \left[4\frac{\Delta t}{(\Delta x)^2}\frac{K}{\eps}\sin^2\left(\frac\theta2\right)\right],\\[5pt]
&\Delta t \hat u^{n+1}(\theta) + \hat Y^{n+1}(\theta) = \hat Y^{n}(\theta). 
\end{aligned}
\right.
\end{equation*}
}
By setting 
$$\alpha_\theta = 1 + 4\frac{\Delta t}{(\Delta x)^2}\left(\mu + \Delta t \frac{K}{\eps}\right)\sin^2\left(\frac\theta2\right), \quad \beta_\theta = 4\frac{\Delta t}{(\Delta x)^2}\frac{K}{\eps}\sin^2\left(\frac\theta2\right), $$
we can recast the system in matrix form 
$$ \underbrace{\begin{pmatrix} \alpha_\theta & 0 \\ \Delta t & 1 \end{pmatrix}}_{A_\theta} \begin{pmatrix} u^{n+1}(\theta) \\ \hat Y^{n+1}(\theta) \end{pmatrix} = \underbrace{\begin{pmatrix} 1 & \beta_\theta \\ 0 & 1 \end{pmatrix}}_{B_\theta} \begin{pmatrix} u^{n}(\theta) \\ \hat Y^{n}(\theta) \end{pmatrix}. $$
The eigenvalues $\lambda_1$ and $\lambda_2$ are the solution of the characteristic polynomial \hfill\break 
${\det(A_\theta^{-1}B_\theta-\lambda\I)=0}$, which is
$$ \det \begin{pmatrix} \frac{1}{\alpha_\theta}-\lambda & \frac{\beta_\theta}{\alpha_\theta} \\ -\frac{\Delta t}{\alpha_\theta} & 1-\Delta t\frac{\beta_\theta}{\alpha_\theta}-\lambda \end{pmatrix} = \alpha_\theta\lambda^2 - \lambda(1+\alpha_\theta-\Delta t\beta_\theta ) + 1 = 0,$$
whose invariants are 
$$ \lambda_1\lambda_2 = \frac1{\alpha_\theta}<1,\quad \lambda_1 + \lambda_2 = \frac{1+\alpha_\theta-\Delta t\beta_\theta}{\alpha_\theta}>0.$$
If $(\lambda_1 + \lambda_2)^2<4\lambda_1\lambda_2$, the eigenvalues are complex conjugates, and their modulus is always less than 1.
If $(\lambda_1 + \lambda_2)^2>4\lambda_1\lambda_2$, the eigenvalues are real and always positive. In particular, 
\begin{align*}
\lambda_1 &= \frac{1+\alpha_\theta-\Delta t\beta_\theta+\sqrt{(1+\alpha_\theta-\Delta t\beta_\theta)^2-4\alpha_\theta}}{2\alpha_\theta} \\
          &\leq \frac{1+\alpha_\theta+\sqrt{(1+\alpha_\theta)^2-4\alpha_\theta}}{2\alpha_\theta} = 1, %\\ 
          %&= \frac{1+\alpha_\theta+\sqrt{(1-\alpha_\theta)^2}}{2\alpha_\theta} = 1,
\end{align*}
which proves that the spectral radius of $A_\theta^{-1}B_\theta$ is always less than one.
\end{proof}

\begin{remark}[Stability of the fully nonlinear semi-implicit schemes]
Observe that the above stability condition has been proven for a linear simplified one-dimensional model and it cannot be generally applied 
to the fully nonlinear fluid-structure problem. In particular, for realistic applications, a relaxation of the time step restriction 
is experienced thanks to the additional dissipation added depending on the local deformation of the media.
\end{remark}
\section{Numerical schemes}\label{sec:schemes}
In this section we present how to develop both explicit and semi-implicit couplings for the full nonlinear model.
\subsection{Explicit method}

In general, The equations (\ref{eq_tot}) are discretized with finite volume schemes on a staggered grid (see Figure~\ref{fig:staggered} for a two-dimensional configuration). 
It should be noticed that the force term $F(\phi,Y)$ is here written as the divergence of a stress tensor multiplied by a cut-off function.

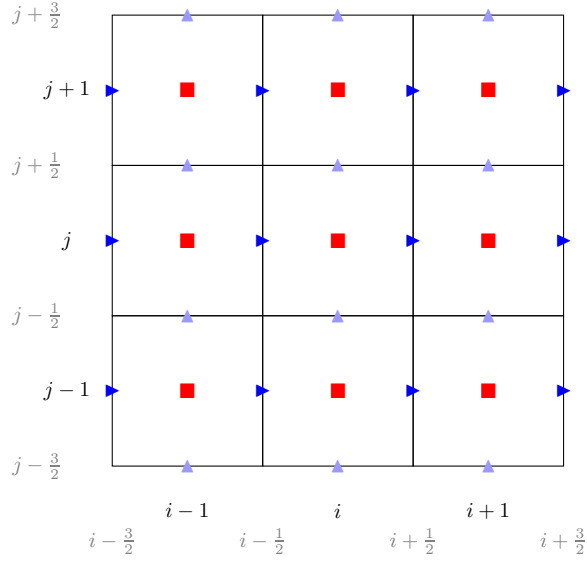
\begin{figure}
\centering
\begin{tikzpicture}[scale=2.0]
    \foreach \x in {0,...,2}
    {
	\foreach \y in {0,...,2}
	{
	    \draw (\x, \y) rectangle (\x+1, \y+1);
	    \node[red,scale=0.8] at (\x+0.5, \y+0.5) {${\color{red}\blacksquare}$};
	}
    }

    \foreach \x in {0,...,2}
    {
	\foreach \y in {0,...,2}
	{
	    \node[scale=0.8] at (\x+0.5, \y) {${\color{blue!40}\blacktriangle}$};
	    \node[scale=0.8] at (\x, \y+0.5) {${\color{blue}\blacktriangleright}$};
	}
    }

    \foreach \x in {3,...,3}
    {
	\foreach \y in {0,...,2}
	{
	    \node[scale=0.8] at (\x, \y+0.5) {${\color{blue}\blacktriangleright}$};
	}
    }

    \foreach \y in {3,...,3}
    {
	\foreach \x in {0,...,2}
	{
	    \node[scale=0.8] at (\x+0.5, \y) {${\color{blue!40}\blacktriangle}$};
	}
    }

    \node[scale=0.8] at ( 0.5, -0.3) {$i-1$};
    \node[scale=0.8] at ( 1.5, -0.3) {$i  $};
    \node[scale=0.8] at ( 2.5, -0.3) {$i+1$};

    \node[scale=0.8] at ( 0.0, -0.5) {${\color{black!50}i-\frac32}$};
    \node[scale=0.8] at ( 1.0, -0.5) {${\color{black!50}i-\frac12}$};
    \node[scale=0.8] at ( 2.0, -0.5) {${\color{black!50}i+\frac12}$};
    \node[scale=0.8] at ( 3.0, -0.5) {${\color{black!50}i+\frac32}$};

    \node[scale=0.8] at ( -0.3, 0.5) {$j-1$};
    \node[scale=0.8] at ( -0.3, 1.5) {$j  $};
    \node[scale=0.8] at ( -0.3, 2.5) {$j+1$};

    \node[scale=0.8] at ( -0.5, 0.0) {${\color{black!50}j-\frac32}$};
    \node[scale=0.8] at ( -0.5, 1.0) {${\color{black!50}j-\frac12}$};
    \node[scale=0.8] at ( -0.5, 2.0) {${\color{black!50}j+\frac12}$};
    \node[scale=0.8] at ( -0.5, 3.0) {${\color{black!50}j+\frac32}$};
\end{tikzpicture}
\caption{Staggered grid with position of unknowns. ${\color{red}\blacksquare}$: scalar field ($p$, $\rho$, $Y$, $\phi$). ${\color{blue}\blacktriangleright}$: vector field -- component along $x$. ${\color{blue!40}\blacktriangle}$: vector field -- component along $y$.}\label{fig:staggered}
\end{figure}

Let $\Delta t$ be the time step and $u^n$, $p^n$, $\phi^n$, $Y^n$, $\rho^n$, $\mu^n$ the time discretization of the variables at $t_n=n\Delta t$. 
The discretization is given by the projection method (see \cite{Guermond2006} for an overview on projection methods) for the Navier Stokes equations and an explicit scheme for the advection equations for the interface and backward characteristics:
%
%
%{\small
\begin{align}
\text{Step 1}:\quad&\rho^n \left(\frac{u^{\star}-u^{n}}{\Delta t} + \div(u^n\otimes u^{\star})\right) + \nabla p^n \label{eq:step1}\\&\qquad\qquad\qquad\qquad\qquad -\div(2\mu^n D(u^{\star}) ) = \delta_\eps(\phi^n) \div(\sigma^n), \nonumber\\
\text{Step 2}:\quad&\div\left(\frac{\Delta t}{\rho^n} \nabla\psi^{n+1}\right) = \div(u^{\star}), \label{eq:step2}\\
\text{Step 3}:\quad&u^{n+1} = u^{\star} - \frac{\Delta t}{\rho^n}\nabla \psi^{n+1},  \qquad  p^{n+1} = p^{n} + \psi^{n+1},  \label{eq:step3} \\
\text{Step 4}:\quad&\frac{\phi^{n+1}-\phi^n}{\Delta t} + u^{n+1}\cdot \nabla \phi^n = 0, \quad \frac{Y^{n+1}-Y^n}{\Delta t} + (u^{n+1}\cdot \nabla) Y^n = 0. \label{eq:step4}
\end{align}
%}
%
%
%
%
In Step 1~\eqref{eq:step1} a prediction of the velocity $u^{\star}$ is computed with an Euler implicit scheme for the viscous term, and an explicit treatment of the elastic source term $\sigma^n$. 
The convection term is treated following a semi-implicit approach, with a standard second-order discretization. 
In Step 2~\eqref{eq:step2}, the Poisson equation for the pressure increment $\psi^{n+1}$ is solved with appropriate boundary conditions~\cite{Guermond2006}. 
The resulting linear systems are solved with the GMRES algorithm of the HYPRE library~\cite{Hypre2002,Hypre2006} with preconditioning. 
In Step 3~\eqref{eq:step3} the velocity is corrected to enforce the incompressibility condition and the pressure is updated. 
In Step 4~\eqref{eq:step4} the transport equations are discretized with an explicit Runge-Kutta3 scheme in time and a WENO5~\cite{shu1996weno} scheme in space for $\phi$ and $Y$. 

For the cut-off function, we considered the following expression:
\begin{equation}
\zeta(r)=
\begin{cases}
 \frac{1}{2}(1+{\rm cos}(\pi r)),\quad \text{on}\quad  [-1,1] \\
 0 ,\quad \text{elsewhere}. 
\end{cases}
\end{equation}
%The Heaviside function is defined as $H(r) = \int_{-\infty}^r \zeta(x) dx$. 
In our simulations $\eps$ is fixed at $2 \Delta x$ which is the standard value used in the literature to spread the interface.

\subsection{Semi-implicit method}\label{sec:semiimplicit}

In the context of fluid-structure interaction problems, a classical explicit treatment of the elastic source term introduces very strict 
time step restrictions, especially when the structure is stiff.
This restriction leads to long and cumbersome simulations, particularly when working with incompressible flows for which long times are 
needed to simulate realistic configurations, and the speed of elastic waves is much faster than the fluid speed.

For such problems, a good approximation of the time step is given by the surface tension condition introduced in~\cite{brackbill1992continuum},
which limits the time step to $\sim\sqrt{\rho\Delta x^3/K}$. 
Although capturing properly all elastic waves becomes crucial for compressible materials in which transients are under study~\cite{de2016cartesian}, in incompressible
simulations one is more interested in the long-time dynamics of deformable structures.
For instance when dealing with very stiff structures, the admissible time step will be extremely low even if the structure will not deform that much. 

For this reason here we present the new semi-implicit method for thin elastic membranes, following the approach introduced in Section~\ref{sec:stability} 
for which it was possible to prove unconditional stability in a simplified framework.
The general idea consists in writing the evolution equation of the stress tensor, 
and then discretizing it in a semi-implicit way with respect to the Eulerian velocity $u$. 
Therefore, we introduce here an evolution equation for the nonlinear membrane stress tensor $\sigma$ defined in~\eqref{eqFm}.
\begin{proposition}\label{prop:evolstress}
Under the smoothness assumption made on $u$, the membrane stress tensor $\sigma$ verifies
\begin{align}\label{eq:evolutionSigmaM}
\pa_t \sigma + &u\cdot\nabla\sigma = \left(f'(Z) Z [\nabla u]:\mathcal{C}\right)\mathcal{C} + \nonumber \\ 
     &f(Z)\left( [\nabla u]^T(n\otimes n) + (n\otimes n)[\nabla u] - 2([\nabla u]n\cdot n)(n\otimes n) \right).
\end{align}
where we defined $f(Z)=E'(Z)Z$, and $\mathcal C = \I - n\otimes n$.
\end{proposition}
\begin{proof}
See Appendix~\ref{app:B}.
\end{proof}

In order to develop the new method for the full model, we start by replacing the explicit term $\sigma^n$ by $\sigma^{\star}$, within the prediction step~\eqref{eq:step1},
\begin{equation}\label{eq:step1SI}
\rho^n \left(\frac{u^{\star}-u^{n}}{\Delta t} + \div(u^n\otimes u^{\star})\right) + \nabla p^n  -\div(2\mu^n D(u^{\star}) )= \delta_\eps(\phi^n) \div(\sigma^{\star}).
\end{equation}
Once one has obtained the evolution equation for $\sigma$~\eqref{eq:evolutionSigmaM}, we discretize it in a semi-implicit way for 
all the terms that depend on $u$ as, 
\begin{align}\label{eq:implicitSigmaM}
\frac{\sigma^{\star} - \sigma^n}{\Delta t} + (u^{\star}\cdot\nabla)\sigma^n =  \mathcal{T}(Z^n,\nabla u^{\star},n) + f(Z^n)\left( [\nabla u^{\star}]^T(n\otimes n) + (n\otimes n)[\nabla u^{\star}] \right)
\end{align}
where
$$ \mathcal{T}(Z^n,\nabla u^{\star},n) = \left(f'(Z^n) Z^n [\nabla u^{\star}]:\mathcal{C}\right)\mathcal{C} - 2 f(Z^n) ([\nabla u^{\star}]n\cdot n)(n\otimes n).$$

When replacing Equation~\eqref{eq:implicitSigmaM} into~\eqref{eq:step1SI} the fluid-structure momentum equation becomes
\begin{align}
&\rho^n \left(\frac{u^{\star}-u^{n}}{\Delta t} + \div(u^n\otimes u^{\star})\right) + \nabla p^n +\delta_\eps(\phi^n) \Delta t\div\bigg[ (u^{\star}\cdot\nabla)\sigma^n - \mathcal{T}(Z^n,\nabla u^{\star},n) \bigg] + \label{eq:SI} \\ 
&\div\bigg[ \biggl(\mu^n\I + \delta_\eps(\phi^n)\Delta t f(Z^n)(n\otimes n) \biggr)[\nabla u^{\star}] + [\nabla u^{\star}]^T\biggl(\mu^n\I + \delta_\eps(\phi^n)\Delta t f(Z^n)(n\otimes n) \biggr) \bigg] \nonumber\\ 
&= \delta_\eps(\phi^n)\div\left(\sigma^n \right). \nonumber
\end{align}
It should be noticed that the consistent terms coming from the evolution equation of $\sigma$ add some viscosity to the classical one.
Moreover, the considered elastic model also gives additional viscosity terms appearing in the $\mathcal{T}$ term of~\eqref{eq:SI}.
All these additional discretized terms depend on the time step $\Delta t$, the geometric features of the membrane described by $n$, and its area variation measured by $Z$.
These new viscosity terms differ significantly from $2\mu D(u)$ because $\mu$ is scalar, while the new terms, like $f(Z^n)(n\otimes n)$, are matrices. For this reason they are going to be referred to as \textit{tensorial viscosity terms}. 
\begin{remark}[Computational costs of the semi-implicit method]
It should be noticed that the resolution of the momentum equation for the explicit scheme already deals with a linear system
to treat implicitly and semi-implicitly the strain rate tensor and the advection term, respectively.  
Therefore, the additional computational cost needed to perform the semi-implicit computations
is given by adding the tensorial viscosity terms to the linear system. 
As shown in Appendix~\ref{app:C} a wider stencil is needed with respect to the classical viscous terms.
These differences do not impact the costs that much also because 
only those elements around the zero level-set require the tensorial viscosity, 
meaning when $\delta_\eps(\phi^n)\ne 0$.
%When $\delta_\eps(\phi^n)= 0$, the system reduces to the Navier-Stokes equations. 
\end{remark}

\section{Manufactured solution}\label{sec:manufsol}

For validation purposes in this section we analyze the convergence properties of the semi-implicit terms, 
discussed in Section~\ref{sec:semiimplicit}, discretized with standard second-order accurate formulas. 
In Appendix~\ref{app:C} we give further details on how to discretize these operators.

The following equation is solved
$$ 10 u + \div\left( (u\cdot\nabla)M - M[\nabla u] - [\nabla u]^T M - ([\nabla u]:M)M \right) = S, $$
which reduces to solving a linear system $\tilde M u_h = S$ in a square domain $[0,1]\times[0,1]$, where $S$ is computed by following the method of manufactured solutions and considering the exact solution

\begin{equation}
u = \begin{bmatrix} \sin\left(\pi x\right)\sin\left(\pi y\right) \\ \sin\left(\pi x\right)\sin\left(\pi y\right) \end{bmatrix} \qquad \text{with} \qquad M = \begin{bmatrix} 1+\sin\left(\pi x y\right) & 1+\sin\left(\pi x y\right) \\ 1+\sin\left(\pi x y\right)  & 1+\sin\left(\pi x y\right) \end{bmatrix}.	
\end{equation}

\begin{remark}[Manufactured solution]
It should be noticed that an additional term $10u$ is added to the diagonal because solving only the semi-implicit operators gives rise to an ill-conditioned linear system which is never considered in practical applications.
As a matter of fact, when solving the Navier-Stokes equation, we always have some additional terms to the diagonal given by the advection and diffusion parts of the system.
This allows us to study the convergence of the discretized operators without having conditioning issues.  
\end{remark}

In Figure~\ref{fig:validationconv} the discretization error is computed with a $L_2$ norm on a set of nested meshes: $20\times 20$, $40\times 40$, $80\times 80$, $160\times 160$, $320\times 320$.
The convergence analysis shows that the discretized operators needed to design the semi-implicit scheme are indeed second-order accurate.

\begin{figure}
\centering
 \begin{tikzpicture}
     \begin{axis}[
      %title = Convergence,
      xmode=log, ymode=log,
      grid=major,
      ylabel={$\|u_h - u_{ex}\|_{L_2}$},
      xlabel={$\Delta x$},
      xlabel shift = 1 pt,
      ylabel shift = 1 pt,
      legend pos=south east,
      legend style={nodes={scale=1.0, transform shape}},
      width=.7\textwidth
      ]
      \addplot[mark=square*,red]     table [y=error, x expr= 1/\thisrow{elem}]{grid_conv_validation.dat};
      \addlegendentry{Numerical results}

      \addplot[black,dashdotted,domain=1/20:1/320]{0.15*x};
      \addlegendentry{first order}
      \addplot[black,dashed,domain=1/20:1/320]{ x^2};
      \addlegendentry{second order}
      \end{axis}
 \end{tikzpicture}
\caption{Validation: convergence analysis with the discretized semi-implicit operators. $u_h$ and $u_{ex}$ represent the numerical and exact solution, respectively.}\label{fig:validationconv}
\end{figure}
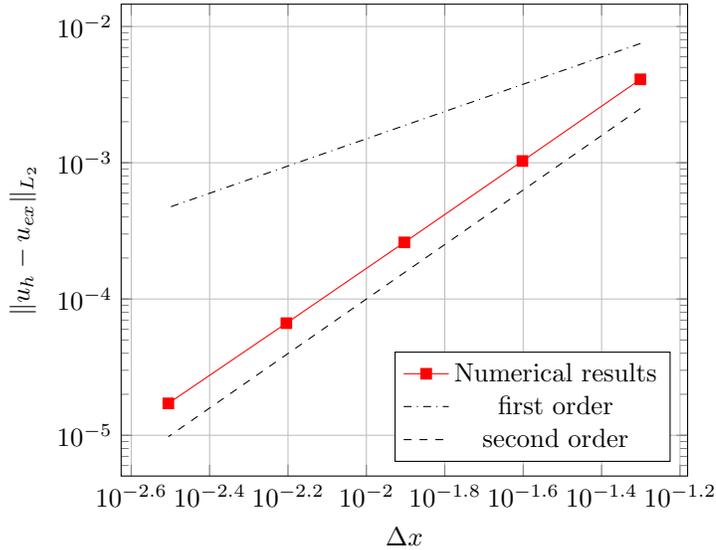

\section{Numerical experiments}\label{sec:simulations}

In this section, we study the effectiveness of the proposed scheme by comparing it to the original explicit one for a thin membrane immersed in a simple shear flow. 
The simulation is setup by considering a membrane of radius $a=0.5$ immersed within a computational domain $[-4,4]\times[-2,2]$ discretized with uniform Cartesian meshes.
From~\eqref{eq:momnondim} the dimensionless parameters are modified by taking the reference speed equal to $a \dot\gamma$, and read
\begin{equation}
\Rey = \frac{\rho a^2 \dot\gamma}{\mu} =0.1, \qquad \Cap = \frac{\mu a \dot\gamma}{K},
\end{equation}
where the capillary number $\Cap$, is modified according to the configuration chosen. 
For all cases, $\rho=\rho_1=\rho_2$ and the angular speed $\dot\gamma$ are both set equal to 1.
From the non-dimensional momentum Equation~\eqref{eq:momnondim} it is possible to notice that when 
$\Cap$ tends to zero, the elastic source term becomes stiff. This also relates to what happens when dealing with low Mach 
flows or materials with the compressible Euler equations~\cite{abbate2017all,abbate2019asymptotic,thomann2023implicit}.
For these experiments, four capillary numbers are considered: 0.02, 0.01, 0.008 and 0.001.
It should be noticed that when dealing with $\Cap=0.001$, the explicit scheme no longer provide any physical result
(when taking $\mu_2/\mu_1=1$) due to the stiff nature of the system, 
and therefore no comparison can be done with the new semi-implicit method.
Therefore, we take the viscosity ratio equal to 10 for the stiffest membrane to perform an adequate comparison, and 1 for the others.
This also shows the ability of the method to deal easily with a viscosity ratio that may differ from one, 
as it may be essential for practical applications~\cite{foessel2011influence}.  
The initial conditions are given by 
\begin{equation}
u_0(x,y) = \begin{bmatrix} \dot\gamma y \\ 0  \end{bmatrix}, \qquad
\phi_{0}(x,y) = \sqrt{x^2 + y^2 } - a,\qquad
Y_0(x,y) = \begin{bmatrix} x \\ y  \end{bmatrix},
\end{equation}
to impose the linear shear velocity field, a circular membrane, and no pre-deformation on the initial configuration.
We set Neumann boundary conditions on the left and on the right of the domain, and moving wall at the top and at the bottom.\\
%
%\begin{table}
%        \caption{Membrane in simple shear flow: comparisons between explicit and semi-implicit for for $\eta=\mu_2/\mu_1=10$ schemes in terms of maximum $\Delta t$ for hard membranes at the maximum elongation point.}\label{tab_membrane2}
%        \centering
%	\begin{tabular}{ccccc} 
%	\hline
%			 &\multicolumn{2}{c}{$\Cap=0.001$}    \Tstrut\Bstrut\\[0.5mm]
%	\cline{2-3}                                         
%	$N_x \times N_y$ & Explicit       & Semi-implicit       \Tstrut\Bstrut\\[0.5mm]\hline
%	$128\times 64 $  & $6.0\;10^{-3}$ &  $6.0\;10^{-2}$     \Tstrut\Bstrut\\
%	$256\times 128$  & $1.0\;10^{-3}$ &  $2.5\;10^{-2}$     \Tstrut\Bstrut\\
%	$512\times 256$  & $2.0\;10^{-4}$ &  $5.0\;10^{-3}$     \Tstrut\Bstrut\\
%        \hline %\hline
%        \end{tabular}
%\end{table}
%
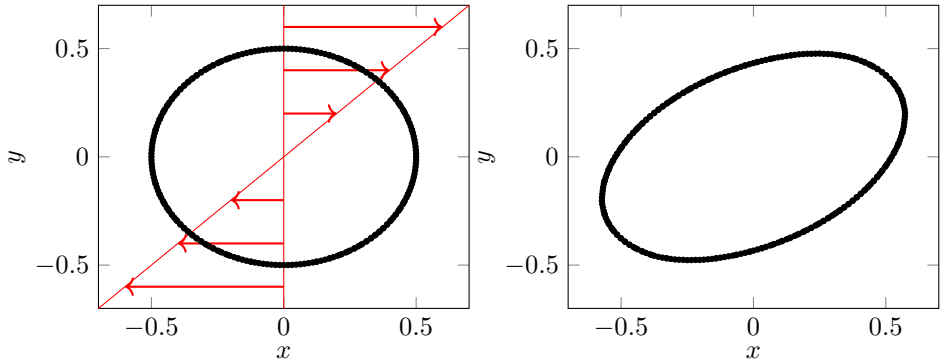
\begin{figure}
\centering
\subfloat{
\begin{tikzpicture}
  \begin{axis}[
    xlabel={$x$},
    ylabel={$y$},
    xmin=-0.7, xmax=0.7,
    ymin=-0.7, ymax=0.7,
    xlabel near ticks,
    ylabel near ticks,
    label style={inner sep=0pt}, %even less space
    width=.50\textwidth
  ]
    \addplot[
	only marks, 
	mark size=1pt,       
	] table [x=x, y=y, col sep=comma] {ShearCa001/fine/shearCa001_init.csv};

    \draw[red](0,0.8)--(0,-0.8);
    \draw[red](0.8,0.8)--(-0.8,-0.8);
    \draw[thick,->,red](0,-0.8)--(-0.8,-0.8);
    \draw[thick,->,red](0,-0.6)--(-0.6,-0.6);
    \draw[thick,->,red](0,-0.4)--(-0.4,-0.4);
    \draw[thick,->,red](0,-0.2)--(-0.2,-0.2);
    \draw[thick,->,red](0, 0.8)--( 0.8, 0.8);
    \draw[thick,->,red](0, 0.6)--( 0.6, 0.6);
    \draw[thick,->,red](0, 0.4)--( 0.4, 0.4);
    \draw[thick,->,red](0, 0.2)--( 0.2, 0.2);

  \end{axis}
\end{tikzpicture}
}
\subfloat{
\begin{tikzpicture}
  \begin{axis}[
    xlabel={$x$},
    ylabel={$y$},
    xmin=-0.7, xmax=0.7,
    ymin=-0.7, ymax=0.7,
    xlabel near ticks,
    ylabel near ticks,
    label style={inner sep=0pt}, %even less space
    width=.50\textwidth
  ]
    \addplot[only marks, mark size=1pt] table [x=x, y=y, col sep=comma] {ShearCa001/fine/shearCa001_final.csv};
  \end{axis}
\end{tikzpicture}
}
\caption{Membrane in simple shear flow: initial and final configurations for soft membranes.}\label{fig:shear2d}
\end{figure}
When a thin membrane is immersed into a shear flow, it starts elongating in the direction of the velocity 
vectors until it reaches a steady state position in which its shape no longer changes but its dynamics is still 
evolving. Indeed, in this configuration, the velocity vectors are perfectly tangent to the membrane and there is a constant stress applied on it, 
which makes it rotating around itself: this phenomenon is called \textit{tank-treading}~\cite{keller1982motion}. 
An overview of the initial and final position of the membrane for this experiment is shown in Figure~\ref{fig:shear2d}.
To compare properly the final shapes of the membrane, the reference time to stop the simulation is set for all cases equal to $1.5$,
which was observed to be enough reach the configuration described above.

As mentioned before, we are interested in studying the dynamics of the membrane when the new semi-implicit scheme is used,
and specifically how the additional viscosity influences the computational cost and the quality of the solution when compared to the original 
explicit scheme. In particular, we focus on the simulation of the tank-treading phenomenon by changing the stiffness coefficient 
of the structure, ranging from a soft membrane to a very stiff one, and performing a grid refinement study to analyze its influence.  

\begin{table}
        \caption{Membrane in simple shear flow: comparisons between explicit (EX) and semi-implicit (SI) schemes in terms of maximum $\Delta t$ to reach a final time $t_f=1.5$.}\label{tab_membrane1}
        \centering
	\begin{adjustbox}{width=\textwidth}
	\begin{tabular}{ccccccccc}
	\hline\hline
	&\multicolumn{2}{c}{$\Cap=0.001$} &\multicolumn{2}{c}{$\Cap=0.008$} &\multicolumn{2}{c}{$\Cap=0.01$} &\multicolumn{2}{c}{$\Cap=0.02$}\Tstrut\Bstrut\\[0.5mm]
	\cline{2-9}
	$N_x \times N_y$ & EX             & SI                &   EX          & SI             & EX             & SI              &  EX            &  SI              \Tstrut\Bstrut\\[0.5mm]\hline
	$128\times 64 $  & $6.0\;10^{-3}$ &  $6.0\;10^{-2}$   & $2.5\;10^{-2}$& $1.0\;10^{-1}$ & $3.0\;10^{-2}$ &  $1.0\;10^{-1}$ & $5.0\;10^{-2}$ & $1.5\;10^{-1}$   \Tstrut\Bstrut\\
	$256\times 128$  & $1.0\;10^{-3}$ &  $2.5\;10^{-2}$   & $1.0\;10^{-2}$& $8.5\;10^{-2}$ & $1.5\;10^{-2}$ &  $9.5\;10^{-2}$ & $2.5\;10^{-2}$ & $1.0\;10^{-1}$   \Tstrut\Bstrut\\
	$512\times 256$  & $2.0\;10^{-4}$ &  $5.0\;10^{-3}$   & $6.0\;10^{-3}$& $6.0\;10^{-2}$ & $5.0\;10^{-3}$ &  $5.0\;10^{-2}$ & $1.0\;10^{-2}$ & $7.5\;10^{-2}$   \Tstrut\Bstrut\\
        \hline\hline
        \end{tabular}
	\end{adjustbox}
\end{table}
In Table~\ref{tab_membrane1} we present the $\Delta t$ needed to perform several experiments
with varying capillary number, from the softest to the stiffest, on a set of three nested uniform meshes $128\times 64 $, $256\times 128$, and $512\times 256$.
It can be noticed that in all cases the semi-implicit method provides higher maximum time steps, proving also numerically speaking that the additional dissipation
introduces consistent viscosity where it is needed allowing us to use larger time steps.
Moreover, the new method provides larger time steps when stiffer structures are under study, which is extremely promising especially for real applications. 
We believe that this numerical behavior can be explained by considering that when dealing with stiffer structure the elastic waves seem  
to become less relevant to capture the general dynamics of the problem. 
This entails larger time steps for stiff structures that do not deform that much, for which the CFL condition given by the elastic waves is very strict.  
Instead, since the deformations are more important for softer membranes, smaller time step ratios seem to be crucial to have good predictions. 

\begin{table}
        \caption{Membrane in simple shear flow: time step ratio $\Delta t_{SI}/\Delta t_{EX}$ with data taken from Table~\ref{tab_membrane1}.}\label{tab_membrane2}
        \centering
	\footnotesize
	\begin{tabular}{ccccc}
	\hline\hline
	                 &$\Cap=0.001$     &$\Cap=0.008$       &$\Cap=0.01$     &$\Cap=0.02$     \Tstrut\Bstrut\\[0.5mm]
	\cline{2-5}
	$N_x \times N_y$      &\multicolumn{4}{c}{$\Delta t_{SI}/\Delta t_{EX}$}                  \Tstrut\Bstrut\\[0.5mm]\hline
	$128\times 64 $  &      10         &        4          &     3.33       &    3            \Tstrut\Bstrut\\
	$256\times 128$  &      25         &        8.5        &     6.33       &    4            \Tstrut\Bstrut\\
	$512\times 256$  &      25         &        10         &     10         &    7.5          \Tstrut\Bstrut\\
        \hline\hline
        \end{tabular}
\end{table}

The mesh refinement analysis also provides interesting results regarding the effect of the new method when refining the mesh.
In particular, it is observed that a larger gain in terms of speedup is experienced on finer grids with respect to coarser grids, 
which is very promising also in view of the future development of the method for three-dimensional models.
Table~\ref{tab_membrane2} shows the ratio between the semi-implicit time step $\Delta t_{SI}$ and the explicit time step $\Delta t_{EX}$, showing its observed increasing trend when using finer grids and stiffer membranes. 
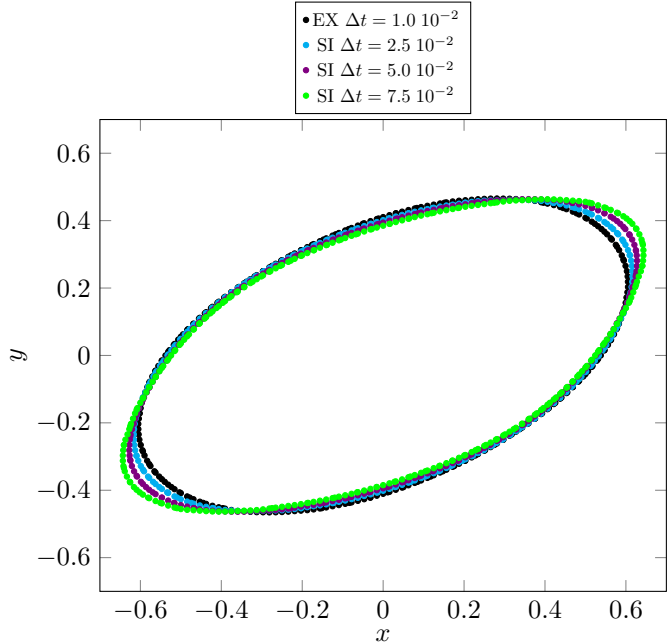
\begin{figure}
\centering
%% \subfloat[Coarse mesh ($128\times64$)]{
%% \begin{tikzpicture}
%%   \begin{axis}[
%%     xlabel={$x$},
%%     ylabel={$y$},
%%     xmin=-0.7, xmax=0.7,
%%     ymin=-0.7, ymax=0.7,
%%     xlabel near ticks,
%%     ylabel near ticks,
%%     label style={inner sep=0pt}, %even less space
%%     width=.50\textwidth,
%%     legend style={nodes={scale=0.7, transform shape},anchor=north,at={(0.5,1.40)}}
%%   ]
%%     \addplot[only marks, mark size=1pt] table [x=x, y=y, col sep=comma]         {ShearCa002/coarse/shearCa002_final.csv};
%%     \addlegendentry{EX $\Delta t=5.0\;10^{-2}$}
%%     \addplot[only marks, mark size=1pt, cyan] table [x=x, y=y, col sep=comma]   {ShearCa002/coarse/shearCa002_SI1.csv};
%%     \addlegendentry{SI $\Delta t=7.0\;10^{-2}$}
%%     \addplot[only marks, mark size=1pt, violet] table [x=x, y=y, col sep=comma] {ShearCa002/coarse/shearCa002_SI2.csv};
%%     \addlegendentry{SI $\Delta t=9.0\;10^{-2}$}
%%     \addplot[only marks, mark size=1pt, green] table [x=x, y=y, col sep=comma]  {ShearCa002/coarse/shearCa002_SI3.csv};
%%     \addlegendentry{SI $\Delta t=1.5\;10^{-1}$}
%%   \end{axis}
%% \end{tikzpicture}
%% }
\subfloat{%[Fine mesh ($512\times256$)]{
\begin{tikzpicture}
  \begin{axis}[
    xlabel={$x$},
    ylabel={$y$},
    xmin=-0.7, xmax=0.7,
    ymin=-0.7, ymax=0.7,
    xlabel near ticks,
    ylabel near ticks,
    label style={inner sep=0pt}, %even less space
    width=.70\textwidth,
    legend style={nodes={scale=0.7, transform shape},anchor=north,at={(0.5,1.25)}}
  ]
    \addplot[only marks, mark size=1pt] table [x=x, y=y, col sep=comma]         {ShearCa002/fine/shearCa002_final.csv};
    \addlegendentry{EX $\Delta t=1.0\;10^{-2}$}
    \addplot[only marks, mark size=1pt, cyan] table [x=x, y=y, col sep=comma]   {ShearCa002/fine/shearCa002_SI1.csv};
    \addlegendentry{SI $\Delta t=2.5\;10^{-2}$}
    \addplot[only marks, mark size=1pt, violet] table [x=x, y=y, col sep=comma] {ShearCa002/fine/shearCa002_SI2.csv};
    \addlegendentry{SI $\Delta t=5.0\;10^{-2}$}
    \addplot[only marks, mark size=1pt, green] table [x=x, y=y, col sep=comma]  {ShearCa002/fine/shearCa002_SI3.csv};
    \addlegendentry{SI $\Delta t=7.5\;10^{-2}$}
  \end{axis}
\end{tikzpicture}
}
\caption{Membrane in simple shear flow: final configuration on the finest mesh for $\Cap=0.02$ with both explicit (EX) and semi-implicit (SI) methods. For the solution computed with the explicit method, the maximum $\Delta t$ shown in Table~\ref{tab_membrane1} is taken. 
For the semi-implicit method, a set of three simulations have been performed with $\Delta t$ ranging from the explicit bound to semi-implicit one shown in Table~\ref{tab_membrane1}.}\label{fig:shear2d_Ca2e-2}
\end{figure}

The increased dissipation experienced for softer structures can also be observed when studying the membrane final shape once the final time is reached.
In the following, we discuss the effect of the numerical scheme on the physics of the problem, which is extremely important to understand if the 
maximum admissible time steps given in Table~\ref{tab_membrane1} are enough to provide a good approximated solution.
For this reason, we perform a set of numerical simulations and compare the final shape obtained through the original explicit scheme with those
simulated with the semi-implicit scheme. Moreover, to investigate further the impact of the time step on the diffusion we compare the explicit result
with a set of three simulations, for each problem, for which the time step has been adapted ranging from the explicit upper bound to the semi-implicit one.
For all cases, the semi-implicit simulations are run with larger time steps with respect to the explicit one.
For instance, in Figure~\ref{fig:shear2d_Ca2e-2} where we show the numerical results for the softest membrane $\Cap=0.02$ on the finest grid, we compare the zero level-set 
of the membrane obtained with the explicit method with $\Delta t=1.0\;10^{-2}$ to the same curve given by the semi-implicit scheme run with $\Delta t=2.5\;10^{-2},\, 5.0\;10^{-2},\, 7.5\;10^{-2}$.
For all experiments we perform the same study and show the results on the finest grid. 
As expected since the proposed scheme is consistent we obtain that all shapes follow the trace of the explicit one.
For the very soft membrane shown in Figure~\ref{fig:shear2d_Ca2e-2} we experience a greater difference in the local error with respect to the explicit solution which increases for larger time steps. 
%This difference seems to be lower when refining the grid even when taking larger time steps, which corroborates the previous statement that for finer meshes the semi-implicit scheme performs better. 
It is observed that, when deformations are important, damping the elastic waves with the tensorial viscosity terms comes with less accurate solutions. 
Moreover, we conclude that for soft structures CFL restrictions are not that important, and therefore the use of more sophisticated methods is not needed. 

\begin{figure}
\centering
%% \subfloat[Coarse mesh ($128\times64$)]{
%% \begin{tikzpicture}
%%   \begin{axis}[
%%     xlabel={$x$},
%%     ylabel={$y$},
%%     xmin=-0.7, xmax=0.7,
%%     ymin=-0.7, ymax=0.7,
%%     xlabel near ticks,
%%     ylabel near ticks,
%%     label style={inner sep=0pt}, %even less space
%%     width=.50\textwidth,
%%     legend style={nodes={scale=0.7, transform shape},anchor=north,at={(0.5,1.4)}}
%%   ]
%%     \addplot[only marks, mark size=1pt] table [x=x, y=y, col sep=comma]         {ShearCa0008/coarse/shearCa0008_final.csv};
%%     \addlegendentry{EX $\Delta t=2.5\;10^{-2}$}
%%     \addplot[only marks, mark size=1pt, cyan] table [x=x, y=y, col sep=comma]   {ShearCa0008/coarse/shearCa0008_SI1.csv};
%%     \addlegendentry{SI $\Delta t=5.0\;10^{-2}$}
%%     \addplot[only marks, mark size=1pt, violet] table [x=x, y=y, col sep=comma] {ShearCa0008/coarse/shearCa0008_SI2.csv};
%%     \addlegendentry{SI $\Delta t=7.5\;10^{-2}$}
%%     \addplot[only marks, mark size=1pt, green] table [x=x, y=y, col sep=comma]  {ShearCa0008/coarse/shearCa0008_SI3.csv};
%%     \addlegendentry{SI $\Delta t=1.0\;10^{-1}$}
%%   \end{axis}
%% \end{tikzpicture}
%% }
\subfloat[$\Cap=0.01$]{
\begin{tikzpicture}
  \begin{axis}[
    xlabel={$x$},
    ylabel={$y$},
    xmin=-0.7, xmax=0.7,
    ymin=-0.7, ymax=0.7,
    xlabel near ticks,
    ylabel near ticks,
    label style={inner sep=0pt}, %even less space
    width=.50\textwidth,
    legend style={nodes={scale=0.7, transform shape},anchor=north,at={(0.5,1.40)}}
  ]
    \addplot[only marks, mark size=1pt] table [x=x, y=y, col sep=comma]         {ShearCa001/fine/shearCa001_final.csv};
    \addlegendentry{EX $\Delta t=5.0\;10^{-3}$}
    \addplot[only marks, mark size=1pt, cyan] table [x=x, y=y, col sep=comma]   {ShearCa001/fine/shearCa001_SI1.csv};
    \addlegendentry{SI $\Delta t=8.0\;10^{-3}$}
    \addplot[only marks, mark size=1pt, violet] table [x=x, y=y, col sep=comma] {ShearCa001/fine/shearCa001_SI2.csv};
    \addlegendentry{SI $\Delta t=2.0\;10^{-2}$}
    \addplot[only marks, mark size=1pt, green] table [x=x, y=y, col sep=comma]  {ShearCa001/fine/shearCa001_SI3.csv};
    \addlegendentry{SI $\Delta t=5.0\;10^{-2}$}
  \end{axis}
\end{tikzpicture}
}
\subfloat[$\Cap=0.008$]{
\begin{tikzpicture}
  \begin{axis}[
    xlabel={$x$},
    ylabel={$y$},
    xmin=-0.7, xmax=0.7,
    ymin=-0.7, ymax=0.7,
    xlabel near ticks,
    ylabel near ticks,
    label style={inner sep=0pt}, %even less space
    width=.50\textwidth,
    legend style={nodes={scale=0.7, transform shape},anchor=north,at={(0.5,1.40)}}
  ]
    \addplot[only marks, mark size=1pt] table [x=x, y=y, col sep=comma]         {ShearCa0008/fine/shearCa0008_final.csv};
    \addlegendentry{EX $\Delta t=6.0\;10^{-3}$}
    \addplot[only marks, mark size=1pt, cyan] table [x=x, y=y, col sep=comma]   {ShearCa0008/fine/shearCa0008_SI1.csv};
    \addlegendentry{SI $\Delta t=9.0\;10^{-3}$}
    \addplot[only marks, mark size=1pt, violet] table [x=x, y=y, col sep=comma] {ShearCa0008/fine/shearCa0008_SI2.csv};
    \addlegendentry{SI $\Delta t=3.0\;10^{-2}$}
    \addplot[only marks, mark size=1pt, green] table [x=x, y=y, col sep=comma]  {ShearCa0008/fine/shearCa0008_SI3.csv};
    \addlegendentry{SI $\Delta t=6.0\;10^{-2}$}
  \end{axis}
\end{tikzpicture}
}
\caption{Membrane in simple shear flow: final configuration on the fine mesh for $\Cap=0.01,\;0.008$ with both explicit (EX) and semi-implicit (SI) methods. For the solution computed with the explicit method, the maximum $\Delta t$ shown in Table~\ref{tab_membrane1} is taken. 
For the semi-implicit method, a set of three simulations have been performed with $\Delta t$ ranging from the explicit bound to semi-implicit one shown in Table~\ref{tab_membrane1}.}\label{fig:shear2d_interm}
\end{figure}
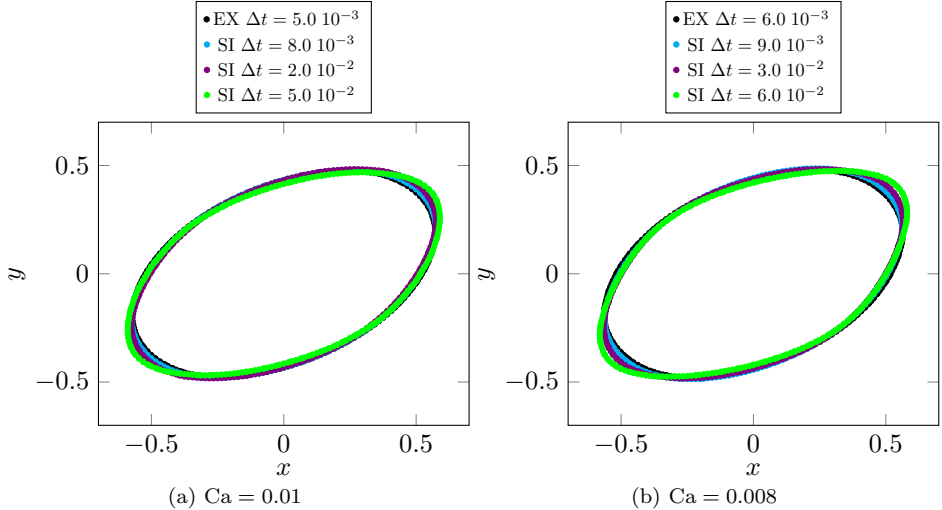
Indeed, it is well-known that stricter restrictions occur for stiffer structures, and it is for those applications that an implicit or semi-implicit approach is crucial.
In Figure~\ref{fig:shear2d_interm}, two intermediate test cases are considered with lower capillary numbers to study the effect of the new scheme in these configurations.
For simplicity, the experiments performed with $\Cap=0.01$ and $\Cap=0.008$ will be discussed together due to similarities in the final shapes, but also 
for the maximum $\Delta t$ and gains shown in Tables~\ref{tab_membrane1} and~\ref{tab_membrane2}. 
For both cases, the semi-implicit scheme introduces much less dissipation giving a final shape of the membrane that is much closer to the reference one with a higher speedup $\Delta t_{SI} /\Delta t_{EX}$, 
especially when compared to the results observed in Figure~\ref{fig:shear2d_Ca2e-2}. 
In particular, in this case a speedup of 10 is observed on the finest mesh. 
\begin{figure}
\centering
%% \subfloat[Coarse mesh ($128\times64$)]{
%% \begin{tikzpicture}
%%   \begin{axis}[
%%     xlabel={$x$},
%%     ylabel={$y$},
%%     xmin=-0.7, xmax=0.7,
%%     ymin=-0.7, ymax=0.7,
%%     xlabel near ticks,
%%     ylabel near ticks,
%%     label style={inner sep=0pt}, %even less space
%%     width=.50\textwidth,
%%     legend style={nodes={scale=0.7, transform shape},anchor=north,at={(0.5,1.4)}}
%%   ]
%%     \addplot[only marks, mark size=1pt] table [x=x, y=y, col sep=comma]         {ShearCa0001/coarse/shearCa0001_final.csv};
%%     \addlegendentry{EX $\Delta t=6.0\;10^{-3}$}
%%     \addplot[only marks, mark size=1pt, cyan] table [x=x, y=y, col sep=comma]   {ShearCa0001/coarse/shearCa0001_SI1.csv};
%%     \addlegendentry{SI $\Delta t=9.0\;10^{-3}$}
%%     \addplot[only marks, mark size=1pt, violet] table [x=x, y=y, col sep=comma] {ShearCa0001/coarse/shearCa0001_SI2.csv};
%%     \addlegendentry{SI $\Delta t=3.0\;10^{-2}$}
%%     \addplot[only marks, mark size=1pt, green] table [x=x, y=y, col sep=comma]  {ShearCa0001/coarse/shearCa0001_SI3.csv};
%%     \addlegendentry{SI $\Delta t=6.0\;10^{-2}$}
%%   \end{axis}
%% \end{tikzpicture}
%% }
\subfloat{%[Fine mesh ($512\times256$)]{
\begin{tikzpicture}
  \begin{axis}[
    xlabel={$x$},
    ylabel={$y$},
    xmin=-0.7, xmax=0.7,
    ymin=-0.7, ymax=0.7,
    xlabel near ticks,
    ylabel near ticks,
    label style={inner sep=0pt}, %even less space
    width=.70\textwidth,
    legend style={nodes={scale=0.7, transform shape},anchor=north,at={(0.5,1.25)}}
  ]
    \addplot[only marks, mark size=1pt] table [x=x, y=y, col sep=comma]         {ShearCa0001/fine/shearCa0001_final.csv};
    \addlegendentry{EX $\Delta t=2.0\;10^{-4}$}
    \addplot[only marks, mark size=1pt, cyan] table [x=x, y=y, col sep=comma]   {ShearCa0001/fine/shearCa0001_SI1.csv};
    \addlegendentry{SI $\Delta t=8.0\;10^{-4}$}
    \addplot[only marks, mark size=1pt, violet] table [x=x, y=y, col sep=comma] {ShearCa0001/fine/shearCa0001_SI2.csv};
    \addlegendentry{SI $\Delta t=2.0\;10^{-3}$}
    \addplot[only marks, mark size=1pt, green] table [x=x, y=y, col sep=comma]  {ShearCa0001/fine/shearCa0001_SI3.csv};
    \addlegendentry{SI $\Delta t=5.0\;10^{-3}$}
  \end{axis}
\end{tikzpicture}
}
\caption{Membrane in simple shear flow: final configuration on the fine mesh for $\Cap=0.001$ with both explicit (EX) and semi-implicit (SI) methods. For the solution computed with the explicit method, the maximum $\Delta t$ shown in Table~\ref{tab_membrane1} is taken. 
For the semi-implicit method, a set of three simulations have been performed with $\Delta t$ ranging from the explicit bound to semi-implicit one shown in Table~\ref{tab_membrane1}.}\label{fig:shear2d_Ca1e-3}
\end{figure}
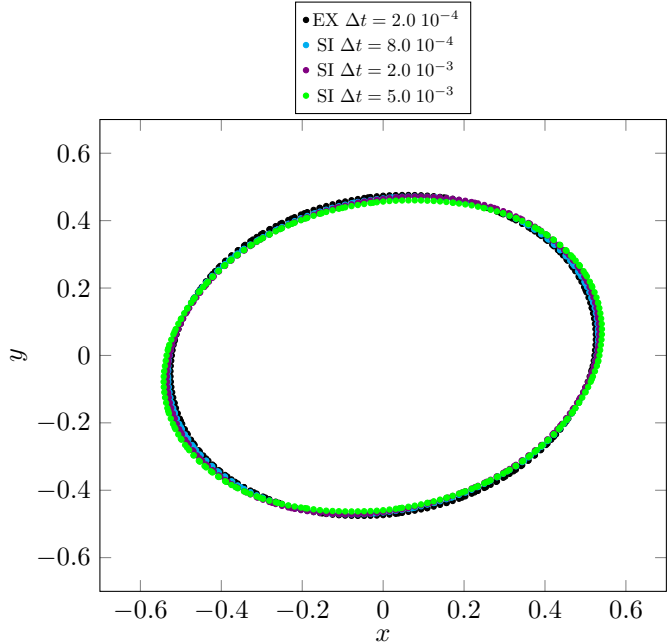
As expected from previous results and assumptions, the best results are obtained for the stiffest membrane simulated with $\Cap=0.001$. 
Figure~\ref{fig:shear2d_Ca1e-3} shows that all semi-implicit results with higher time step are almost superimposed on the reference shape computed with $\Delta t=2.0\;10^{-4}$.
In particular, for this experiment, we observe a speedup of 10 on the coarse mesh and 25 on the fine one that translates into a dramatic reduction of the computational time
which also comes with a very precise prediction of the membrane shape.
Although elastic waves are extremely fast for this experiment, the additional viscosity seem to be able to dampen them and relax the time step restriction,
still providing a good dynamics of the membrane with almost no additional costs.

We conclude by pointing out that the proposed approach behaves the best for the stiffest structure on the finest mesh.
This is critical for both real applications that often concern stiff materials and computational efficiency, which one would like to improve especially when refining the mesh.

\section{Summary and future perspectives}\label{sec:conclusions}

We presented a novel approach for the construction of semi-implicit schemes for Eulerian models for fluid-structure interaction.
This approach is very interesting for the simulation of incompressible flows in which the speed of elastic waves is much higher than that of the fluid
causing the need of using extremely low time steps when compared to the simulation time.
The focus of this work was on hyperelastic membranes and on the two-dimensional version of the model for full membrane elasticity introduced in~\cite{Milcent2016}. 
The tensorial viscosity terms provide additional but consistent dissipation that activates only on the
few cells where the membrane interface is spread. 
This was also demonstrated for a simplified one-dimensional model in Section~\ref{sec:stability}, where this additional term allowed us to prove that the scheme is unconditionally stable.
As shown in Section~\ref{sec:schemes}, in the full nonlinear fluid-structure system these consistent terms appear in the momentum equation~\eqref{eq:SI} and depend on the time step, the stiffness parameter, and the local deformation of the membrane. 
In Section~\ref{sec:simulations} a set of numerical experiments is performed for a membrane immersed in a simple shear flow. 
It was shown that the dependences mentioned before influence the performances of the new scheme, which is observed to be more efficient and accurate for stiffer structures and finer grids. 
%This was explained by the fact that stiffer structures deform less, making less important the need of capturing properly all elastic waves that can be dampened without losing too much accuracy.
In particular, the numerical method proposed herein performed best on the experiment with the lowest capillary number showing a remarkable speedup
of 10 on the coarse mesh and 25 on the fine one while preserving a very good agreement with the reference solution.

The perspectives of this work are several and range from the development of the method for other hyperelastic models to its application to advance simulations of biological systems.
 
Regarding the investigations on the method, in this work we considered 
the two-dimensional version of the full membrane model introduced in~\cite{Milcent2016}, 
and future studies will also focus on the development of the method for the full three-dimensional model. 
Following the same reasoning, new evolution equations will be written to take into account the modeling of shear variation. 
Therefore, this will introduce viscosity terms that depend on both the area and shear variation.
Other models will also be considered to study the effect of this approach on surface tension (special case of the membrane model with area variation) and, 
anisotropic elastic bodies~\cite{cottet2008eulerian}.   

About the possible applications, the proposed method is very promising for the simulation of biological capsules and red blood cells (RBCs), 
modelled as thin membranes characterized by shear effects, and their interaction with complex geometries~\cite{Hu2013,haner2021sorting}.

\appendix

\section{$Z$ is the local area variation}\label{app:A}

In this appendix, we prove that $Z$ introduced in~\eqref{Z1_Z2b} measures the local area variation as 
it was firstly introduced in~\cite{cottet2006level}. In particular, in~\cite{cottet2006level} it was proven that 
for incompressible media the gradient of the level-set $|\nabla\phi|$ transports information about the local area variation. 
Afterwards, in~\cite{Milcent2016} a new theory was introduced with the goal of having a single formulation 
(for both compressible and incompressible media) to deal with both area and shear variations. 
For this reason, two parameters have been introduced for three-dimensional membranes as functions only of the backward characteristics $Y$.

Here, we prove that to capture the area variation for two-dimensional membranes a different parameter is needed with
respect to~\cite{Milcent2016}, and we also show how to relate it to the theory introduced in~\cite{cottet2006level}. 

\begin{proposition}

Given that $J=\det(\nabla_\xi X) = \det(\nabla Y)^{-1} $ is the volume ratio, different than one for general compressible media, 
\begin{equation}\label{appeq:Z}
Z = J \frac{|\nabla\phi|}{|\nabla\phi_{0}(Y)|}. 
\end{equation}
\end{proposition}
\begin{proof}
The local area variation defined in~\eqref{Z1_Z2b} is measured by 
\begin{equation}\label{appeq:Zeq}
Z = \sqrt{\text{Tr}(\mathcal{A})}. 
\end{equation}
When using definition~\eqref{AB}, we can write

\begin{align}\label{appeq:TrA}
\text{Tr}(\mathcal{A}) &= \text{Tr}(B) - \frac{\text{Tr}((Bn)\otimes(Bn))}{(Bn)\cdot n} %\nonumber \\
                       = \frac{((\text{Tr}(B)B - B^2)n)\cdot n}{(Bn)\cdot n} %\nonumber \\
                       = \frac{\det(B)}{(Bn)\cdot n}, 
\end{align}

where, for $B \in \mathcal M_2(\R)$, the Cayley-Hamilton theorem $B^2-\text{Tr}(B)B+\text{det}(B)\I=0$ was used in the last step.\\
With Equation~\eqref{normales}, and $\nabla\phi=[\nabla Y]^T \nabla\phi_{0}(Y)$, we can write

\begin{align}\label{appeq:Bnn}
 (Bn)\cdot n &= \left([\nabla Y]^{-1} [\nabla Y]^{-T} \frac{\nabla\phi}{|\nabla\phi|} \right)\cdot \frac{\nabla\phi}{|\nabla\phi|} %\nonumber\\ 
 	     = \left| [\nabla Y]^{-T} \frac{\nabla\phi}{|\nabla\phi|} \right|^2 %\nonumber\\
	     = \frac{|\nabla\phi_{0}(Y)|^2}{|\nabla\phi|^2}. 
\end{align}
With the relation $\det(B)=J^2$ and \eqref{appeq:Bnn}, Equation~\eqref{appeq:Zeq} reduces to~\eqref{appeq:Z}.
\end{proof}

\begin{remark}[Volume ratio]
It should be noticed that Equation~\eqref{appeq:Zeq} measures the area variation for both compressible and incompressible flows.
In particular, for incompressible flows, $J$ is exactly equal to one and $Z = |\nabla\phi|/|\nabla\phi_{0}(Y)|$ as it was introduced in~\cite{cottet2006level}. 
\end{remark}

\section{Proof of Proposition~\ref{prop:evolstress}}\label{app:B}

\begin{proof*}
By performing the derivative in time of $\sigma$ we obtain
\begin{equation}\label{eq:dersigma}
\pa_t \sigma = \pa_t f(Z) \,\mathcal{C} + f(Z) \pa_t\mathcal{C}.
\end{equation}
We first focus on the equation related to $f(Z)$. Considering the time derivative of $f(Z)$ gives us
\begin{equation}\label{eq:evolutionf}
\pa_t f(Z) + u\cdot\nabla f(Z) = f'(Z) Z\, [\nabla u]:\mathcal{C},
\end{equation}
where the last step is obtained with the equation on $Z$ taken from~\cite{Milcent2016}, recalled here
\begin{equation}\label{eq:evolutionZ}
\pa_t Z + u\cdot\nabla Z = Z\, [\nabla u]:\mathcal{C}.
\end{equation}
To compute the equation on $\mathcal C$, an equation on $n$ is needed. The evolution of $n$ is computed by deriving Equation~\eqref{normales} as follows
\begin{equation*}
\pa_t n = \frac{1}{|\nabla\phi|} \left[ \pa_t(\nabla\phi) - \left( \pa_t(\nabla\phi)\cdot\frac{\nabla\phi}{|\nabla\phi|}\right)\frac{\nabla\phi}{|\nabla\phi|} \right]. 
\end{equation*}
This relation together with the evolution equation on the gradient of $\phi$~\eqref{AB} 
\begin{equation*}
\pa_t (\nabla\phi) + u\cdot\nabla(\nabla\phi) = -[\nabla u]^T\nabla\phi,
\end{equation*}
gives us
\begin{equation}\label{eq:evolutionN}
\pa_t n + u\cdot\nabla n = -[\nabla u]^T n + (([\nabla u]n)\cdot n)n.
\end{equation}
With the identities $(A b)\otimes c = A (b\otimes c)$ and $b\otimes (Ac) = (b \otimes c)A^T$, it is possible to write the equation on $\mathcal C$, which reads
\begin{equation}\label{eq:evolutionC}
\pa_t \mathcal C + u\cdot\nabla \mathcal C = [\nabla u]^T (n\otimes n) + (n\otimes n) [\nabla u] -2 (([\nabla u]n)\cdot n)(n\otimes n).
\end{equation}
Using \eqref{eq:evolutionC} and \eqref{eq:evolutionf} in~\eqref{eq:dersigma} we get the evolution equation~\eqref{eq:evolutionSigmaM} stated in the proposition. 
\end{proof*}

\section{Discretization of semi-implicit operators}\label{app:C}

In this appendix, we briefly describe how we discretize the semi-implicit operators for the staggered mesh framework used by the incompressible Navier-Stokes solver.
In particular, for a two-dimensional problem, up to 16 terms need to be discretized: 
$\partial_{i} (m \, \partial_{j} u_k)$, with $i,j,k=x,y$, for both $u_x$ and $u_y$ located on vertical and horizontal faces, respectively (see Figure \ref{fig:staggered}). 
Indeed, since the position of the $x-$ and $y-$components of a vector field are located in different points the discretization of
the same operator on one direction is going to be different to that on the other direction.
For compactness, here we do not present the discretization of all terms but, we mainly focus on few of them to give an idea on how to discretize these operators.  
There are of course several ways to discretize such terms but as long as they are consistent no remarkable difference is observed.  
For simplicity a uniform discretization is considered, meaning that each cell has a $\Delta x \times \Delta y$ area. 
To simplify the notation, only in this appendix we take the following abuse of notation: $u=u_x$ and $v=u_y$.\\
 
For both components, we discretize two trivial operators and a third trickier one.
For the $x-$direction (on vertical faces) we consider the discretization of $\partial_{x} (m \, \partial_{x} u)$, $\partial_{y} (m \, \partial_{y} u)$, and $\partial_{y} (m \, \partial_{y} v)$:
{\small
\begin{align*}
\partial_x (m\, \partial_x u)|_{i-1/2,j} &= \frac1{\Delta x} \left( m_{i,j}           \frac{ u_{i+1/2,j}     - u_{i-1/2,j} }{\Delta x} - m_{i-1,j}         \frac{ u_{i-1/2,j}   - u_{i-3/2,j}   }{\Delta x} \right),  \\ 
\partial_y (m\, \partial_y u)|_{i-1/2,j} &= \frac1{\Delta y} \left( m^*_{i-1/2,j+1/2} \frac{ u_{i-1/2,j+1}   - u_{i-1/2,j} }{\Delta y} - m^*_{i-1/2,j-1/2} \frac{ u_{i-1/2,j}   - u_{i-1/2,j-1} }{\Delta y}  \right), \\ 
\partial_y (m\, \partial_y v)|_{i-1/2,j} &= \frac1{\Delta y} \left( m^*_{i-1/2,j+1/2} \frac{ v^*_{i-1/2,j+1} - v^*_{i-1/2,j} }{\Delta y} - m^*_{i-1/2,j-1/2} \frac{ v^*_{i-1/2,j} - v^*_{i-1/2,j-1} }{\Delta y} \right),
\end{align*}
}
where $(\cdot)^*$ identifies an interpolated variable. In this case, we can approximate the multiplicative factor $m$ as
\begin{align*}
& m^*_{i-1/2,j+1/2} = \frac14(m_{i,j}+m_{i,j+1}+m_{i-1,j+1}+m_{i-1,j}),\\
& m^*_{i-1/2,j-1/2} = \frac14(m_{i,j-1}+m_{i,j}+m_{i-1,j}+m_{i-1,j-1}), 
\end{align*}
and the velocity component $v$ as
\begin{align*}
& v^*_{i-1/2,j+1} = \frac14(v_{i,j+1/2}+v_{i,j+3/2}+v_{i-1,j+3/2}+v_{i-1,j+1/2}), \\
& v^*_{i-1/2,j} = \frac14(v_{i,j-1/2}+v_{i,j+1/2}+v_{i-1,j+1/2}+v_{i-1,j-1/2}),\\
& v^*_{i-1/2,j-1} = \frac14(v_{i,j-3/2}+v_{i,j-1/2}+v_{i-1,j-1/2}+v_{i-1,j-3/2}).
\end{align*}
For the $y-$direction (on horizontal faces) we consider the discretization of $\partial_{y} (m \, \partial_{y} v)$, $\partial_{x} (m \, \partial_{x} v)$,  and $\partial_{x} (m \, \partial_{x} u)$:
\begin{align*}
\partial_y (m\, \partial_y v)|_{i,j-1/2} &= \frac1{\Delta y} \left( m_{i,j}           \frac{v_{i,j+1/2}  -v_{i,j-1/2}}{\Delta y} - m_{i,j-1}         \frac{v_{i,j-1/2}-v_{i,j-3/2  }}{\Delta y} \right), \\
\partial_x (m\, \partial_x v)|_{i,j-1/2} &= \frac1{\Delta x} \left( m^*_{i+1/2,j-1/2} \frac{v_{i+1,j-1/2}-v_{i,j-1/2}}{\Delta x} - m^*_{i-1/2,j-1/2} \frac{v_{i,j-1/2}-v_{i-1,j-1/2}}{\Delta x} \right)   , \\ 
\partial_x (m\, \partial_x u)|_{i,j-1/2} &= \frac1{\Delta x} \left( m^*_{i+1/2,j-1/2} \frac{ u^*_{i+1,j-1/2} - u^*_{i,j-1/2} }{\Delta x} - m^*_{i-1/2,j-1/2} \frac{ u^*_{i,j-1/2} - u^*_{i-1,j-1/2} }{\Delta x} \right),
\end{align*}
where we can approximate the remaining $m$ term as
$$ m^*_{i+1/2,j-1/2} = \frac14(m_{i+1,j-1}+m_{i+1,j}+m_{i,j}+m_{i,j-1}) , $$
and the velocity component $u$ as
\begin{align*}
&u^*_{i+1,j-1/2} = \frac14(u_{i+3/2,j-1}+u_{i+3/2,j}+u_{i+1/2,j}+u_{i+1/2,j-1}), \\ 
&u^*_{i,j-1/2} = \frac14(u_{i+1/2,j-1}+u_{i+1/2,j}+u_{i-1/2,j}+u_{i-1/2,j-1}), \\
& u^*_{i-1,j-1/2} = \frac14(u_{i-1/2,j-1}+u_{i-1/2,j}+u_{i-3/2,j}+u_{i-3/2,j-1}).
\end{align*}

\section*{Acknowledgments}

MC would like to acknowledge the support of Antoine Lemoine for the development of this project within the open-source 
massively parallel software {\bf notus CFD} (\url{https://notus-cfd.org/}).

\bibliographystyle{siamplain}
\bibliography{bib}
\end{document}